\documentclass[11pt,leqno]{amsart}
\usepackage{amsmath,amssymb,amsthm}
\usepackage{hyperref}
\usepackage{xcolor}
\usepackage{graphicx}
\usepackage[matrix,arrow]{xy}

\usepackage{tikz-cd}
\usepackage{caption}

\usepackage{float}

\usepackage{bbm}

\newcommand{\R}{\mathbb{R}}

\renewcommand{\geq}{\geqslant}
\renewcommand{\leq}{\leqslant}

\newcommand{\Lin}{{\mathcal L}}

\newtheorem{theorem}{Theorem}[section]
\newtheorem{lemma}[theorem]{Lemma}

\newtheorem{proposition}[theorem]{Proposition}
\newtheorem{corollary}[theorem]{Corollary}
\newtheorem{question}[theorem]{Question}
\theoremstyle{definition}

\theoremstyle{remark}
\newtheorem{remark}[theorem]{Remark}
\numberwithin{equation}{section}

\def\fnote#1{\footnote}

\def\natu{{\mathbb N}}

\def\ignora#1{}
\def\n3#1{\left\vert \! \left\vert \! \left\vert \, #1 \, \right\vert \!
 \right\vert \! \right\vert }

\begin{document}

\title{Rank-one perturbations and norm-attaining operators}

\author[Jung]{Mingu Jung}
\address[Jung]{School of Mathematics, Korea Institute for Advanced Study, 02455 Seoul, Republic of Korea \newline
\href{http://orcid.org/0000-0003-2240-2855}{ORCID: \texttt{0000-0003-2240-2855} }}
\email{\texttt{jmingoo@kias.re.kr}}
\urladdr{\url{https://clemg.blog/}}

\author[Mart\'inez-Cervantes]{Gonzalo Mart\'inez-Cervantes}
\address[G. Mart\'inez-Cervantes]{Universidad de Alicante, Departamento de Matem\'{a}ticas, Facultad de Ciencias, 03080 Alicante, Spain
	\newline
	\href{http://orcid.org/0000-0002-5927-5215}{ORCID: \texttt{0000-0002-5927-5215} } }	
\email{gonzalo.martinez@ua.es}

\author[Rueda Zoca]{Abraham Rueda Zoca}
\address[Rueda Zoca]{Universidad de Granada, Facultad de Ciencias.
Departamento de An\'{a}lisis Matem\'{a}tico, 18071-Granada, Spain \newline
\href{https://orcid.org/0000-0003-0718-1353}{ORCID: \texttt{0000-0003-0718-1353} }} \email{abrahamrueda@ugr.es}
\urladdr{\url{https://arzenglish.wordpress.com}}



\maketitle

\begin{abstract}

The main goal of this article is to show that for every (reflexive) infinite-dimensional Banach space $X$ there exists a reflexive Banach space $Y$ and $T, R \in \mathcal{L}(X,Y)$ such that $R$ is a rank-one operator, $\|T+R\|>\|T\|$ but $T+R$ does not attain its norm. This answers a question posed by S.~Dantas and the first two authors.
Furthermore, motivated by the parallelism exhibited in the literature between the $V$-property introduced by V.A.~Khatskevich, M.I.~Ostrovskii and V.S.~Shulman and the weak maximizing property introduced by R.M. Aron, D.~Garc\'ia, D. Pellegrino and E.V.~Teixeira, we also study the relationship between these two properties and norm-attaining perturbations of operators.

\end{abstract}

\section{Introduction}

Given real Banach spaces $X$ and $Y$, we denote by $\mathcal{L}(X,Y)$ (resp., $\mathcal{K}(X,Y)$) the space of all (bounded linear) operators (resp., compact operators) from $X$ into $Y$. When $X=Y$, we simply write $\mathcal{L}(X)$ (resp., $\mathcal{K}(X)$). As usual, the notations $B_X$ and $S_X$ stand for the closed unit ball of $X$ and the unit sphere of $X$, respectively.

A well-known result of J. Lindenstrauss states that if $X$ is reflexive then, for every Banach space $Y$, every operator $T\in \mathcal{L}(X,Y)$ can be approximated by norm-attaining operators. Recall that an operator $T \in \mathcal{L}(X,Y)$ is said to be \emph{norm-attaining} if $\|T\| = \|T(x)\|$ for some $x \in B_X$. Indeed, J. Lindenstrauss showed
that there exists a compact operator $K \in \mathcal{K}(X,Y)$ such that $\|K\|$ is arbitrarily small and $T+K$ attains its norm (see Theorem 1 and the succeeding remark in \cite{Lin1}), so every operator can be approximated by norm-attaining compact perturbations of itself.

Pairs of classical Banach spaces quite often satisfy a stronger condition. Namely, J.~Kover \cite{Kover} proved that for a Hilbert space $H$, if a compact perturbation of an operator $T \in \mathcal{L}(H)$ has norm strictly greater than the norm of $T$, then this perturbation attains its norm, i.e.~if $\|T+K\|>\|T\|$ where $K \in \mathcal{K}(H)$ then $T+K$ attains its norm.

Before going beyond in the exposition of literature, let us fix the following notation: a pair $(X,Y)$ of Banach spaces has the \emph{compact perturbation property} (for short, CPP) if for any $T \in \mathcal{L} (X,Y)$ and $K \in \mathcal{K} (X,Y)$ the inequality $\|T +K \| > \|T\|$ implies that $T+K$ is norm-attaining. Notice that the CPP of the pair $(X,Y)$ forces the domain space $X$ to be reflexive. 
We say that $X$ has the CPP if the pair $(X,X)$ has the CPP. With this notation in mind, Kover's result says nothing but that every Hilbert space has the CPP.


A large class of pairs of Banach spaces enjoying the CPP is given by the class of $V$-pairs, which was introduced and intensively developed in the papers \cite{KOS,O}: a bounded linear operator $T\in \mathcal{L}(X,Y)$ is said to be a \textit{$V$-operator} if there is a norm-one operator $S \in \mathcal{L}(Y,X)$ such that the spectral radius of $TS$ coincides with $\|T\|$. If every operator in $\mathcal{L}(X,Y)$ is a $V$-operator, then the pair $(X,Y)$ is said to be a \textit{$V$-pair} or to have the \textit{$V$-property}. If $X=Y$, then $X$ is said to be a $V$-space or to have the $V$-property.
Among others, it is proved in \cite[Proposition 5]{KOS} that an operator having a strictly singular hump is a $V$-operator if and only if it is norm-attaining. Thus, every $V$-pair satisfies the CPP. Moreover, \cite[Theorem 1]{KOS} generalizes the aforementioned result of Kover to $\ell_p (X_n)$ for any sequence of finite dimensional spaces $(X_n)$ and $1<p<\infty$.

Quite recently R.M. Aron, D.~Garc\'ia, D. Pellegrino and E.V.~Teixeira \cite{AGPT} introduced another related property, the so-called {weak maximizing property}. A pair $(X,Y)$ of Banach spaces is said to have the \emph{weak maximizing property} (for short, WMP) if every operator from $X$ into $Y$ with a non-weakly null maximizing sequence is norm-attaining. Note from \cite[Proposition 2.2]{AGPT} (see also \cite[Theorem 1]{PelTeix}) that for $1<p<\infty$, $1\leq q <\infty$, and arbitrary index sets $\Gamma_1, \Gamma_2$, the pair $(\ell_p (\Gamma_1), \ell_q (\Gamma_2))$ has the WMP. Furthermore, the WMP implies the CPP \cite[Proposition 2.4]{AGPT}, which improves the previous results of Kover. Although there are pairs of Banach spaces with the CPP failing the WMP (see \cite[Proposition 3.6]{djm}), it seems to be an open problem whether the CPP for a pair of \emph{reflexive} Banach spaces implies the WMP \cite[Question 4.4]{djm}. For further results and more examples of pairs with the WMP we refer the reader to \cite{djm, gp}.

Due to the aforementioned results it is natural to wonder whether every reflexive Banach space $X$ has the CPP. Nevertheless, it is implicitly proved in \cite[Theorem 2]{O} that for every infinite-dimensional Banach space $X$ there exists an equivalent norm $||| \cdot |||$ and a rank-one operator $R:X \longrightarrow X$ with $||| I+R ||| > ||| I |||$ but with $I+R$ failing to attain its norm in $\mathcal{L}((X, ||| \cdot |||))$. In particular, $(X, ||| \cdot |||)$ fails the CPP (see Proposition \ref{prop:Ost}).

Note that if a Banach space $X$ is such that every operator in $\Lin (X)$ attains its norm, then $X$ clearly has the CPP. In other words, if $X$ fails the CPP, then there exists a non-norm attaining operator on $X$. By the result just mentioned in \cite{O}, we conclude that for any infinite-dimensional Banach space $X$, there exists a renorming $\widetilde{X}$ of $X$ for which not every operator in $\Lin(\widetilde{X})$ is norm-attaining. In fact, it seems to be an open problem whether there exists a reflexive infinite-dimensional Banach space $X$ such that every operator in $\mathcal{L}(X)$ attains its norm (see, for instance, \cite[Problem 8]{KOS}).

The main goal of this paper is to show that for every infinite-dimensional space $X$ there exists a reflexive Banach space $Y$ such that the pair $(X,Y)$ fails the CPP. Namely, the main theorem of the paper reads as follows.

\begin{theorem}\label{maintheorem}
	Let $X$ be an infinite-dimensional reflexive Banach space. Then there exists a reflexive Banach space $Y$ and $T,R\in \mathcal{L}(X,Y)$ with $R$ a rank-one operator such that $\|T+R\|>\|T\|$ but with $T+R$ not attaining its norm. In particular, the pair $(X,Y)$ fails the CPP.
\end{theorem}

%
%

The aim of Section \ref{section:main} is to prove Theorem \ref{maintheorem}. Our original motivation for Theorem \ref{maintheorem} comes from the study of the WMP and the $V$-property. There is a deep parallelism exhibited in the literature between these two properties; compare \cite{djm, gp} to \cite{KOS, O}. However, up to our knowledge, it is not known how these two concepts are related to each other and quite often if a question is open for one property it is also open for the other. In \cite[Question 4.3]{djm} it is asked whether if a reflexive Banach space $X$ satisfies that the pair $(X,Y)$ has the WMP for every Banach space $Y$, then $X$ must be finite-dimensional. We would like to mention that, to the best of our knowledge, the same question was open if we replace the WMP with the CPP.

In particular, Theorem \ref{maintheorem} gives a positive answer to \cite[Question 4.3]{djm} and exhibits another common behaviour between the WMP and the $V$-property.

Observe that the WMP and the $V$-property implies the CPP. We know that the WMP (and therefore the CPP) does not imply the $V$-property (the pair $(\ell_p, \ell_q)$ with $1<p<q<\infty$ has the WMP while it is not a $V$-pair). Nevertheless, we do not know whether the $V$-property implies the WMP. This situation also  serves as motivation for our investigation of all the aforementioned properties.
In Section \ref{sec:relationproperties}, we first observe from the argument of Ostrovskii \cite{O} that the CPP is still a quite restrictive property: an infinite-dimensional Banach space $X$ with the CPP must be isometric to $\ell_p$ for some $1<p <\infty$ if $X$ has a symmetric basis (Proposition \ref{prop:Ost}). 
Moreover, we see that the existence of non-norm-attaining operators between Banach spaces $X$ and $Y$ produces pairs of Banach spaces failing the CPP (Proposition \ref{prop:djm}) and, as a consequence, the pair $(L_p[0,1], L_q [0,1])$ fails the CPP whenever $p>2$ or $q<2$ (Corollary \ref{cor:LpLq}). We also generalize in Proposition \ref{prop:generalizationWMP} the fact that the pair $(\ell_p,\ell_q)$ has the WMP for $1<p<\infty$ and $1\leq q <\infty$. Finally, we prove that a pair $(X,Y)$ has the $V$-property if and only if every operator from $X$ into $Y$ is norm-attaining provided that $Y$ has the Dunford-Pettis property (Proposition \ref{prop:DPP}), which covers some results in previous sections.

\section{Proof of Theorem \ref{maintheorem}}\label{section:main}

As we have indicated in the introduction, the aim of this section is to prove Theorem \ref{maintheorem}, for which we will start by considering the Banach space $\ell_\infty$ of all bounded sequences instead of reflexive ones.

\begin{theorem}\label{thm:main}
Let $X$ be a reflexive Banach space. 
If the pair $(X,\ell_\infty)$ has the CPP, then $X$ is finite-dimensional. Namely, if $X$ is infinite-dimensional then there exists $S,R \in \mathcal{L}(X,\ell_\infty)$ with $R$ a rank-one operator such that $\|S+R\|>\|S\|$ but with $S+R$ not attaining its norm.
\end{theorem}

For the proof we need a lemma of geometric nature for points of Fr\'echet differentiability. Let us recall that a point $x\in S_X$ is said to be a \textit{point of Fr\'echet-differentiability of $X$} if the norm $\Vert\cdot\Vert: X\longrightarrow \mathbb R$ is Fr\'echet differentiable at $x$. See \cite[Chapter 8]{checos} for background about Fr\'echet differentiability in Banach spaces.

\begin{lemma}\label{lem:F-point}
Let $X$ be a Banach space. Let $x\in S_X$ be a point of Fr\'echet differentiability of $X$ and $(x_n)$ be a sequence of points of $B_X$ such that $\|x -x_n\| \rightarrow 2$. Then $\limsup\limits_{n \rightarrow \infty} \|x+x_n\| < 2$. 
\end{lemma}

\begin{proof}
Since $x$ is a point of Fr\'echet differentiability, by \v{S}mulyan Lemma \cite[Lemma 8.4]{checos} there exists $f\in S_{X^*}$ with $f(x)=1$ and with the following property: for every $\varepsilon>0$ there exists $\delta>0$ such that if $g\in B_{X^*}$ and $g(x)>1-\delta$ then $\Vert g-f\Vert<\varepsilon$.

Assume to the contrary that $\limsup\limits_{n \rightarrow \infty} \| x + x_n \| =2$. Take a subsequence $(x_{n_k})$ such that $\| x + x_{n_k} \| \rightarrow 2$. 
 
Fix $\varepsilon>0$ and take the $\delta >0$ associated to $f$ above, which we can assume to satisfy $\varepsilon< 1-\delta$ (observe that if $\delta$ satisfies the above condition, any $\delta'<\delta$ will also satisfy the condition).

As $\| x - x_{n_k}\| \rightarrow 2$, we can pick $k \in \mathbb{N}$ such that $\| x \pm x_{n_k} \| > 2 - \delta$. 
Find $f^\pm\in B_{X^*}$ such that $f^\pm (x\pm x_{n_k})>2-\delta$. This implies that $f^\pm (x)>1-\delta$ and $\pm f^\pm(x_{n_k})>1-\delta$. 

Observe, on the one hand, that $\Vert f^+-f^-\Vert\geq (f^+-f^-)(x_{n_k})>2(1-\delta)$. On the other hand, Smulyan test implies that $\Vert f^+-f^-\Vert< 2 \varepsilon$, so $1-\delta<\varepsilon$, a contradiction. 
\end{proof}

\begin{proof}[Proof of Theorem \ref{thm:main}]
	We suppose by contradiction that the Banach space $X$ is infinite-dimensional. Let $x_0 \in S_X$ be a point of Fr\'echet differentiability of $B_X$ (such point exists because of the reflexivity of $X$ \cite[Corollary 11.10]{checos}) and let $x_0^* \in S_{X^*}$ such that $x_0^* (x_0)=1$. 
	
	Set $Y:=\ker(x_0^*)$. Since $Y$ is $1$-codimensional we get that $Y$ is infinite-dimensional, so Josefson-Nissenzweig Theorem guarantees the existence of a weak$^*$-null sequence $(y_n^*)$ in $ S_{Y^*}$. 
	Set a norm-one extension $x_n^*\in S_{X^*}$ of $y_n^*$ given by Hahn-Banach Theorem for every $n\in \natu$. Since $y_n^*\in S_{Y^*}$ we can find $x_n\in S_Y$ such that $y_n^*(x_n)=1$. In particular, $x_n^*(x_n)=y_n^*(x_n)=1$ for every $n\in\mathbb N$.
	
	If $\| x_0 - x_n \| \rightarrow 2$, then using Lemma \ref{lem:F-point}, we have $\limsup\limits_{n \rightarrow \infty} \| x_0+ x_n \| < 2$. Otherwise, passing to a subsequence, we may assume that $\sup\limits_{n \in \mathbb{N}} \|x_0 - x_n\| < 2$. Thus, in any case, we may assume by passing to a subsequence that either $\sup\limits_{n \in \mathbb{N}} \| x_0 + x_{n}\| < 2$ or $\sup\limits_{n \in \mathbb{N}} \| x_0 - x_{n} \| < 2$. Set $\epsilon\in \{-1,+1\}$ such that  $\sup\limits_{n \in \mathbb{N}} \| x_0 + \epsilon x_{n} \| < 2$.

	Passing again to a further subsequence if necessary, we can suppose that $(x_n^*(x_0))$ converges to some $\alpha \in \R$. Since every element of $X$ is of the form $ax_0+y$ with $y\in Y=\ker(x_0^*)$, it follows that $(x_n^*-\alpha x_0 ^*)$ is weak$^*$-null.


	Define now $g_n:=(1-\epsilon\alpha)x_0^*+\epsilon x_n^*$ for every $n\in\mathbb N$. Observe that 
	\[
	\lim\limits_{n \rightarrow \infty} g_n (x_0) = \lim\limits_{n \rightarrow \infty} (1-\epsilon\alpha)x_0^*(x_0)+\epsilon x_n^*(x_0)=1-\epsilon\alpha+\epsilon\alpha=1. 
	\]
	Moreover, since $x_n\in Y=\ker(x_0^*)$, we have $x_0^*(x_n)=0$ for every $n\in\mathbb N$ and therefore   
	\[
	\lim\limits_{n \rightarrow \infty} g_n (\epsilon x_n) =\lim\limits_{n \rightarrow \infty} ( \epsilon(1-\epsilon\alpha) x_0^* (x_n)+\epsilon^2 x_n^*(x_n))=1.
	\]

	Thus,  
	\[
	L:= \limsup\limits_{n \rightarrow \infty} \Vert g_n \Vert\geq \frac{ \lim\limits_{n \rightarrow \infty} g_n (x_0 + \epsilon x_n) }{ \sup_{n \in \mathbb{N}} \Vert x_0+ \epsilon x_n \Vert} = \frac{2}{ \sup_{n \in \mathbb{N}} \Vert x_0+ \epsilon x_n \Vert} >1. 
	\]

	For every $n\in\mathbb N$ take $v_n\in S_{X}$ such that $g_n(v_n)>\Vert g_n\Vert-\frac{1}{n}$. Take a subsequence $(g_{n_k})$ with $\| g_{n_k}\| \rightarrow L$. Note that $g_{n_k}(v_{n_k})\rightarrow L$ as $k \rightarrow \infty$. Passing to a further subsequence, we may assume that $| g_{n_k} (v_{n_k}) - L | < \frac{1}{k}$ for every $k \in \mathbb{N}$. 
	Notice that for each $k \in \mathbb{N}$ 
	\[
	L - \frac{1}{k} < g_{n_k} (v_{n_k}) \leq \| g_{n_k} \| \leq g_{n_k} (v_{n_k}) + \frac{1}{n_k} \leq g_{n_k} (v_{n_k}) + \frac{1}{k} < L + \frac{2}{k}. 
	\]
	This, in particular, shows that $\|g_{n_k}\| \neq 0$. Put $\alpha_k := \|g_{n_k} \|^{-1} (L - \frac{1}{k})$ for every $k \in \mathbb{N}$; then $\alpha_k \in (0,1)$ and $\alpha_k \rightarrow 1$ as $k \rightarrow \infty$.

	Define $T \in \Lin (X, \ell_\infty)$ by the equation
	$$T(x):=(\alpha_k g_{n_k} (x))_{k\in \mathbb N}$$ 
	and observe that $\Vert T\Vert\leq L$ because $\alpha_k \|g_{n_k}\| = L- \frac{1}{k} < L$ for every $k \in \mathbb{N}$.
	On the other hand, observe that $\Vert T(v_{n_k})\Vert\geq \alpha_{k}g_{n_k}(v_{n_k})>\alpha_{k}( L-\frac{1}{k})$, which implies that $\Vert T\Vert=L$.

	We claim that $T$ does not attain its norm. Assume to the contrary that there exists $u_0 \in S_X$ such that $\Vert T(u_0 )\Vert=\sup\limits_{k\in \mathbb N} \vert \alpha_k g_{n_k} (u_0 )\vert = L$. 
	Observe that 
	\[
	g_n=(1-\epsilon\alpha)x_0^*+\epsilon x_n^*=x_0^*+\epsilon (x_n^*-\alpha x_0^*) \rightarrow x_0^* 
	\]
	in the weak$^*$-topology since $(x_n^*-\alpha x_0^*)$ is weak$^*$-null. Consequently $g_n(u_0)\rightarrow x_0^* (u_0)$ and, since $\vert x_0^* (u_0)\vert\leq 1$,  we can find $k_0 \in\mathbb N$ such that $\vert g_{n_k} (u_0)\vert< \frac{L+1}{2}<L$ for $k\geq k_0$.
	This would imply that $\|T \| = \max\limits_{1 \leq k \leq k_0 -1} \{ \alpha_k \| g_{n_k} (u_0)\| \} = L$. However, $\alpha_k \| g_{n_k} (u_0) \| \leq \alpha_k \|g_{n_k}\| = L - \frac{1}{k} < L$ for each fixed $k$, which leads to a contradiction.

	Now notice that $$T(x):=(\alpha_k g_{n_k} (x))_{k\in \mathbb N}=(\alpha_k(1-\epsilon\alpha)x_0^*(x)+\epsilon\alpha_k x_n^*(x))_{k\in \mathbb N}=R(x)+S(x),$$
	where $S(x)=(\epsilon \alpha_k x_n^*(x))_{k\in \mathbb N}$ and $R(x) =(1-\epsilon\alpha)x_0^*(x) (\alpha_k)_{k\in \mathbb N}$ is a rank-one operator 
	and $$\|S\|=1 < \|S+R\|=\|T\|=L ,$$
	which concludes the proof.
\end{proof}

\begin{remark}\label{remark:c}
As a consequence of Theorem \ref{thm:main}, if $X$ is an infinite-dimensional reflexive Banach space, then the pair $(X, \ell_\infty)$ has neither the $V$-property nor the WMP. Notice that the argument in Theorem \ref{thm:main} also applies to the pair $(X, c)$. That is, $(X, c)$ does not have the CPP unless $X$ has finite dimension. 
Nevertheless, the pair $(X,c_0)$ has the WMP for any reflexive Banach space $X$ \cite[Proposition 3.6]{djm} while $c_0$ and $c$ are isomorphic. Thus, the {CPP} is not an \textit{isomorphic} property. 
\end{remark}

Now we prove that we can replace $\ell_\infty$ with a suitable reflexive Banach space $Y$ in Theorem \ref{thm:main}.


\begin{proof}[Proof of Theorem \ref{maintheorem}]
	By Theorem \ref{thm:main}, there exist $T \in \mathcal{L} (X, \ell_\infty)$ and $K \in \mathcal{K} (X, \ell_\infty)$ such that $\|T + K \| > \|T\| = 1$ while $T+K$ does not attain its norm. Actually, $K$ can be chosen to be a rank-one operator. 
	Consider $\widetilde{T}$ and $\widetilde{K}$ in $\mathcal{L} (X, \ell_\infty \oplus_\infty X )$ given by 
	\[
	\widetilde{T} (x) = (T(x), x), \quad \widetilde{K} (x) = (K(x), 0)  \text{ for every } x \in X. 
	\]
	Note that $\| \widetilde{T} + \widetilde{K} \| = \|T + K \| > \|T \| = \| \widetilde{T} \|$, $\widetilde{K}$ is a rank one operator, and $\widetilde{T} + \widetilde{K}$ does not attain its norm. 
	Let $Z:= \overline{\text{span}}\{ \widetilde{T} (X) \cup \widetilde{K} (X) \} \subseteq \ell_\infty \oplus_\infty X$. Observe that $\widetilde{T}$ has a closed range which implies that $\widetilde{T} (X)$ is isomorphic to a quotient space of $X$; hence $\widetilde{T} (X)$ is reflexive. It follows that $Z= \overline{\text{span}}\{ \widetilde{T} (X) \cup \widetilde{K} (X) \} =\widetilde{T} (X) \oplus \widetilde{K} (X)$ is a reflexive Banach space since $\widetilde{T} (X)$ and $\widetilde{K} (X)$ are reflexive.
%
%
%
%
%
	Considering $\widetilde{T} + \widetilde{K}$ as an operator from $X$ into $Z$, we conclude that the pair $(X,Z)$ fails the CPP.
\end{proof}

\section{Interrelations between the WMP, $V$-property and CPP}\label{sec:relationproperties}

The aim of this section is to make an intensive study of the WMP, $V$-pairs and the CPP. 
As it was said in the introduction, it is known that the CPP is the weakest one among all the above properties. However, the CPP itself is still very restrictive, as the following two results shows.

It is worth mentioning that Ostrovskii \cite[Theorem 2]{O} proved that for any infinite-dimensional Banach space $X$, there exists a renorming $\widetilde{X}$ of $X$ such that in $\widetilde{X}$ there is a projection onto a subspace of codimension $1$ that does not attain its norm (in particular, with norm strictly bigger than one). From this, we can conclude that $\widetilde{X}$ does not have the CPP since there is a rank-one operator $Q$ on $\widetilde{X}$ such that $I_{\widetilde{X}} - Q$ does not attain its norm. He also proved that if $X$ is a Banach space with a symmetric basis and has the $V$-property, then $X$ is isometric to $\ell_p$ for some $1<p<\infty$ \cite[Theorem 3]{O}. In fact, his argument shows the following: if $X$ is a Banach space with a symmetric basis and has the CPP, then $X$ is isometric to $\ell_p$. We summarize these comments in the following proposition. 

\begin{proposition}\label{prop:Ost}
Let $X$ be an infinite-dimensional Banach space.
\begin{enumerate}
\itemsep0.3em
\item Then there is a renorming $\widetilde{X}$ of $X$ which does not have the CPP.
\item If $X$ has a symmetric basis and has the CPP, then $X$ is isometric to $\ell_p$ for some $1 < p <\infty$. 
\end{enumerate} 
\end{proposition}

Another manifestation of the severe restriction that CPP on a pair $(X,Y)$ imposes on the spaces $X$ and $Y$ is that we can always find a pair of Banach spaces which fails to have the CPP from the existence of non-norm-attaining operators as follows. It is an analogue of \cite[Main Theorem]{djm}. 

\begin{proposition}\label{prop:djm}
Let $X$ and $Y$ be Banach spaces, and suppose that there exists a non-norm-attaining operator $T \in \mathcal{L}(X,Y)$. 
Then the pair $(X \oplus_p \mathbb{R}, Y \oplus_q \mathbb{R})$ fails to have the CPP whenever $1 \leq q < p \leq \infty$. 
\end{proposition} 

\begin{proof}
Let $T \in \mathcal{L}(X,Y)$ be a non-norm-attaining operator with $\|T\| = 1$. Assume first that $1 \leq q < p < \infty$. Consider $\widetilde{T}, R \in \mathcal{L} (X \oplus_p \R, Y \oplus_q \R)$ given by $\widetilde{T}(x,a)= (T(x), 0)$ and $R(x,a)= (0, a)$, respectively. It is clear that $R$ is compact (indeed, of finite rank) and $\|\widetilde{T}\| = \|T\|$.
Moreover, by \cite[Lemma 2.1]{djm} and the argument of the proof of \cite[Main Theorem]{djm}, we have that 
\begin{align*}
\|\widetilde{T} + R\| = \sup \{ ( (1-t^p)^{q/p} + t^q )^{1/q} : t \in [0,1] \} > 1 = \|\widetilde{T}\|,
\end{align*}
where the hypothesis that $1 \leq q < p < \infty $ is used. However, it is not difficult to check that $\widetilde{T}+R$ is not norm-attaining; hence the pair $(X \oplus_p \mathbb{R}, Y \oplus_q \mathbb{R})$ fails to have the CPP. In the case $p = \infty$ we have that $\|\widetilde{T} + R\| = 2^{1/q} > 1 = \|\widetilde{T}\|$, so the conclusion holds. 
\end{proof}

\begin{remark}
Under the same hypothesis, it is observed in \cite[Main Theorem]{djm} that the pair $(X \oplus_\infty \mathbb{R}, Y)$ fails to have the WMP. However, we cannot expect the same result for the case of the CPP. In fact, there exists a non-norm-attaining operator from $X$ into $c_0$ whenever $X$ is an infinite-dimensional Banach space $X$ \cite[Lemma 2.2]{MMP}, while $(X,c_0)$ has the CPP for any reflexive Banach space $X$ \cite[Proposition 3.6]{djm}. 
\end{remark} 

\begin{remark}
Observe that the Schur property of a Banach space $Y$ implies that $(X,Y)$ has the CPP for every reflexive space $X$ (in fact, it has the WMP \cite[Theorem 3.5]{djm}). However, the converse is not true as the space $c_0$ does not have the Schur property. 
\end{remark} 

As an application of Proposition \ref{prop:djm}, the pair $(L_p [0,1], L_q [0,1])$ fails the CPP whenever $p>2$ or $q<2$. To verify this we need the following lemma, which is a version of \cite[Proposition 2.2]{djm} and \cite[Proposition 4]{KOS} for the CPP. 

\begin{lemma}\label{lem:basic}
Let $X$ and $Y$ be Banach spaces. Suppose that the pair $(X,Y)$ has the CPP. Then for any subspaces $X_1 \subseteq X$ and $Y_1 \subseteq Y$, the pair $(X/X_1, Y_1)$ has the CPP. 
\end{lemma} 

\begin{proof}
Let $\pi : X \rightarrow X/X_1$ be the canonical quotient map and $\iota$ be the inclusion from $Y_1$ into $Y$. Suppose that $T \in \mathcal{L}(X/X_1, Y_1)$ and $K \in \mathcal{K} (X/X_1, Y_1)$ satisfy that $\|T+K \| > \|T\|$. Then it is clear that $\| \iota T \pi + \iota K \pi \| = \|T+K \| > \|T\| = \| \iota T \pi \|$. By the assumption, the perturbation $\iota T \pi + \iota K \pi$ attains its norm at some $x_0 \in B_X$. From this, we have that $T+K$ attains its norm at $\pi(x_0)$; hence the pair $(X/X_1, Y_1)$ has the CPP. 
\end{proof} 

Arguing as in the proof of \cite[Theorem 3.2]{djm} with the aid of Lemma \ref{lem:basic}, we obtain the following desired result. 

\begin{corollary}\label{cor:LpLq}
If $p>2$ or $q<2$, then the pair $(L_p [0,1], L_q [0,1])$ fails the CPP. 
\end{corollary}




Note that the pair $(\ell_p, \ell_q)$ with $1<p<q<\infty$ has the WMP (and therefore the CPP) while it is not a $V$-pair.
We do not know whether the reverse implication holds for every pair of Banach spaces:
\begin{question}
Does every $V$-pair have the WMP?
\end{question}

A natural idea to prove that every $V$-pair has the WMP would be to try to show that every $V$-operator with a non-weakly null maximizing sequence is norm-attaining. 
However, the following remark shows that there are  $V$-operators with non-weakly null maximizing sequences which do not attain their norm.

\begin{remark}
\label{RemarkVoperator}
Set $X=\ell_2 \oplus_\infty \R$. Take a non-norm-attaining operator $T \in \mathcal{L}(\ell_2)$ with $\|T\|=1$ and define $\widehat{T} \in \mathcal{L}(X, \ell_2)$ given by $\widehat{T}(x,a)=T(x)$. It is immediate that $\widehat{T}$ does not attain its norm (which is one) and that it has a non-weakly null maximizing sequence (take $((x_n,1))$ with $(x_n)$ a maximizing sequence for $T$). Nevertheless, we claim that $\widehat{T}$ is a $V$-operator. Since $T$ is a $V$-operator, there exists a norm-one operator $B \in \mathcal{L}( \ell_2)$ such that $TB-I$ is not bounded below.
Define $\widehat{B} \in \mathcal{L}( \ell_2, X)$ such that $\widehat{B}(x)=(B(x),0)$. It is immediate that $\widehat{T}\widehat{B}-I$ is not bounded below and therefore $\widehat{T}$ is a $V$-operator as desired.

Despite the fact that $\widehat{T}$ is a $V$-operator, Proposition \ref{prop:djm} implies that the pair $(X,\ell_2)$ does not have the CPP (so, it is not a $V$-pair). 
\end{remark}

In spite of the fact that the properties WMP and being a $V$-pair are not equivalent properties on a pair $(X,Y)$ of Banach spaces, they are closely related properties (as both imply, for instance, the CPP) and, because of that, it is expectable that some results which holds true for $V$-pairs can be translated to a WMP version and vice-versa. This is what we will do in the following proposition. 

Observe that, for each $1$-complemented subspace $X_1$ of $X$, the pair $(X_1, Y)$ has the WMP (resp., is a $V$-pair) whenever $(X,Y)$ has the WMP (resp., is a $V$-pair) (see \cite[Proposition 2.2]{djm}, resp., \cite[Proposition 4]{KOS}). 
However, it is unknown if the same remains true for any subspace of the domain space $X$. In this regard, the following result extends \cite[(c) of Corollary 3.8]{gp}, where the authors proved the WMP of a pair $(X,Y)$ for $X = (\sum_{n=1}^\infty E_n)_p$ and $Y = (\sum_{n=1}^\infty F_n)_q$, where $\text{dim}(E_n), \text{dim}(F_n) < \infty$, $1 < p< \infty$ and $1 \leq q < \infty$.

\begin{proposition}\label{prop:generalizationWMP}
Let $(E_n)$ and $(F_n)$ be sequences of finite-dimensional spaces. 
If $X \subseteq (\sum_{n=1}^\infty E_n)_p$ and $Y \subseteq (\sum_{n=1}^\infty F_n)_q$ for $1 < p< \infty$ and $1 \leq q < \infty$, then the pair $(X,Y)$ has the WMP. 
\end{proposition} 

\begin{proof}
The proof is motivated by the result of Ostrovskii \cite[Lemma 1]{O}. 
Let $T \in \Lin (X,Y)$, $\|T\| =1$ and let $(x_n) \subseteq S_X$ be a maximizing sequence for $T$. Suppose that $(x_n)$ converges weakly to $x_0 \neq 0$. We aim to show that $T$ attains its norm.

Case 1: $1<p\leq q < \infty$. For simplicity, consider the norm $\| \cdot \|_r$ on $\mathbb{R}^2$ given by $\|(a,b)\|_r := (|a|^r + |b|^r)^{1/r}$ for $1 < r< \infty$. Then 
\begin{enumerate}
\itemsep0.3em
\item For a weakly null sequence $(w_n)$ in $X$ and $v \in X$, we have 
\[
\| w_n + v \| = \| (\| w_n\| , \|v\| ) \|_p + o(1) 
\]
\item $\| (a,b)\|_r < \| (c,d)\|_r$ if $0\leq a \leq c$ and $0\leq b \leq d$, and at least one of the inequalities is strict. 
\end{enumerate} 
Put $w_n = x_n - x_0$. Then $(w_n)$ is weakly null (so, $(T(w_n))$ is weakly null in $Y$); hence 
\begin{align*}
\|T(x_n) \| = \| T(w_n) + T(x_0) \| &= \| ( \|T(w_n)\|, \|T(x_0)\| )\|_q + o(1) \\
&\leq \| ( \| w_n\|, \|T(x_0)\| )\|_q + o(1). 
\end{align*} 
On the other hand,
\[
\|T(x_n)\| = \| w_n + x_0 \| + o(1) = \| ( \|w_n\|, \|x_0\| ) \|_p + o(1). 
\]
Passing to a subsequence, $\|w_n\| \rightarrow \alpha$. Thus, we have 
\[
\| ( \alpha, \|x_0\| ) \|_p \leq \| ( \alpha , \|T(x_0)\| )\|_q \leq \| ( \alpha , \|T(x_0)\| )\|_p.
\]
This shows that $\|T(x_0) \| = \|x_0\|$. Since $x_0 \neq 0$, we conclude that $T$ attains its norm at $x_0$. 

Case 2: $1 \leq q < p < \infty$. Let us denote by $\overline{\delta}_Z (t)$ and $\overline{\rho}_Z (t)$ the modulus of AUC and AUS of a Banach space $Z$, respectively. Note from \cite[Theorem 3.1]{Milman} that 
\begin{align*}
\overline{\rho}_X (t) \leq \overline{\rho}_{ (\sum_{n=1}^\infty E_n)_p } (t) &= (1+t^p)^{1/p} - 1 \\
&< (1+t^q)^{1/q} - 1 = \overline{\delta}_{(\sum_{n=1}^\infty F_n)_q} (t) \leq \overline{\delta}_Y (t)
\end{align*} 
for $0<t<1$. By \cite[Proposition 2.3]{JLPS}, we conclude that every operator from $X$ into $Y$ is compact; hence $T$ attains its norm. 
\end{proof}


As we already observed from Theorem \ref{thm:main} and Remark \ref{remark:c}, the pairs $(X, c)$ and $(X, \ell_\infty)$ cannot be $V$-pairs (since they fail to have the CPP). This is also covered by the following result since $c$ and $\ell_\infty$ (in general, $C(K)$-spaces) have the Dunford-Pettis property. Recall that a Banach space $X$ is said to have the \emph{Dunford-Pettis property} if every weakly compact operator from $X$ into any Banach space is completely continuous.

\begin{proposition}\label{prop:DPP}
Let $X$ be a Banach space and $Y$ a Banach space with Dunford-Pettis property. If the pair $(X,Y)$ is a $V$-pair, then every operator in $\mathcal{L}(X,Y)$ is norm-attaining.
\end{proposition} 

\begin{proof}
Let $T \in \Lin (X,Y)$ with $\|T\| =1$. As $(X,Y)$ is a $V$-pair, there is $B \in \Lin (Y,X)$ with $\|B\| =1$ and a sequence $(y_n) \subseteq S_Y$ such that $\|(TB - I)(y_n) \| \rightarrow 0$. Since $X$ is reflexive, passing to a subsequence, we may assume that $(B(y_n))$ converges weakly to some $x_0 \in B_X$. 
It follows that $(TB(y_n))$ converges weakly to $T(x_0)$ in $Y$, which in turn implies that $(y_n)$ converges weakly to $T(x_0)$. As $Y$ has the Dunford-Pettis property, $B$ is completely continuous; hence $(B(y_n))$ converges in norm to $BT(x_0)$. 
Note then that $(TB(y_n))$ converges in norm to $TBT(x_0)$; hence, $(y_n)$ converges in norm to $TBT(x_0)$. This shows that $\|TBT(x_0)\| = 1$ and $T$ attains its norm at $x_0$.
\end{proof}

Unlike the WMP, it follows from Proposition \ref{prop:DPP} that for any infinite-dimensional Banach space $X$, the pair $(X, c_0)$ is not a $V$-pair, which stresses the fact that the implication ``WMP $\Rightarrow$ $V$-property" is false.

\bigskip

\noindent\textbf{Acknowledgment.} 
The research of this paper started during a stay of the authors in IMAC Castell\'on in June 2022. The authors are deeply grateful to the whole institute, and very specially to Sheldon Dantas, for the hospitality received during this stay.

Mingu Jung was supported by a KIAS Individual Grant (MG086601) at Korea Institute for Advanced Study. 
The last two authors were partially supported by Agencia Estatal de Investigación and EDRF/FEDER ``A way of making Europe" (MCIN/AEI/10.13039/501100011033) through grants PID2021-122126NB-C32 and PID2021-122126NB-C31 (Rueda Zoca). The research of Abraham Rueda Zoca was also supported by Junta de Andaluc\'ia  Grants FQM-0185 and PY20\_00255.

\end{document}